\documentclass[12pt]{amsart}
\usepackage[initials]{amsrefs}
\usepackage{amssymb,graphicx,amscd,color,amsmath,amsfonts,amssymb,latexsym,amsthm,units,paralist}
\usepackage{geometry}
\usepackage[all]{xy}

\newtheorem{theorem}{Theorem}[section]
\newtheorem{proposition}[theorem]{Proposition}
\newtheorem{corollary}[theorem]{Corollary}

\newtheorem{lemma}[theorem]{Lemma}

\newtheorem{question}[theorem]{Question}
\newtheorem{definition}[theorem]{Definition}
\numberwithin{equation}{section}

\begin{document}

\title{Basic sequences and spaceability in $\ell_p$ spaces}

\author[Cariello \and Seoane]{Daniel Cariello \and Juan B. Seoane-Sep\'{u}lveda}

\address{Departamento de An\'{a}lisis Matem\'{a}tico,\newline\indent Facultad de Ciencias Matem\'{a}ticas, \newline\indent Plaza de Ciencias 3, \newline\indent Universidad Complutense de Madrid,\newline\indent Madrid, 28040, Spain.}
\email{dcariello@famat.ufu.br, jseoane@mat.ucm.es}

\thanks{}


\maketitle

\begin{abstract}
Let $X$ be a sequence space and denote by $Z(X)$ the subset of $X$ formed by sequences having only a finite number of zero coordinates. We study algebraic properties of $Z(X)$ and show (among other results) that (for $p \in [1,\infty]$) $Z(\ell_p)$ does not contain infinite dimensional closed subspaces. This solves an open question originally posed by R. M. Aron and V. I. Gurariy in 2003 on the linear structure of $Z(\ell_\infty)$. 

In addition to this, we also give a thorough analysis of the existing algebraic structures within the set $X \setminus Z(X)$ and its algebraic genericity.
\end{abstract}

\section{Introduction and Preliminaries}

During a {\em Non-linear Analysis Seminar} at Kent State University (Kent, Ohio, USA) in 2003, Richard M. Aron and Vladimir I. Gurariy posed the following question:
\begin{question}[R. Aron \& V. Gurariy, 2003]\label{Q} \mbox{}\newline
Is there an infinite dimensional closed subspace of $\ell_{\infty}$ every nonzero element of which has only a finite number of zero coordinates?
\end{question}

Using modern terminology (originally coined by V. Gurariy himself), a subset $M$ of a topological vector space $X$ is called {\em lineable} (resp. {\em spaceable}) in $X$ if there exists an infinite dimensional linear space (resp. an infinite dimensional {\em closed} linear space) $Y \subset M\cup \{0\}$ (see \cite{AGS,BAMS,enflogurariyseoane2012,GQ}). V. Gurariy also coined the notion of {\em algebrability} (introduced in \cite{aronseoane2007}) meaning that, given a Banach algebra $\mathcal{A}$ and a subset $\mathcal{B} \subset \mathcal{A}$, it is said that $\mathcal{B}$ is {\em algebrable} if there exists a subalgebra $\mathcal{C}$ of $\mathcal{A}$ so that $\mathcal{C} \subset \mathcal{B} \cup \{ 0\}$ and the cardinality of any system of generators of $\mathcal{C}$ is infinite. The links between the previous concepts are as follows (all the implications in the previous diagram are strict, see e.g., \cite{BAMS}).

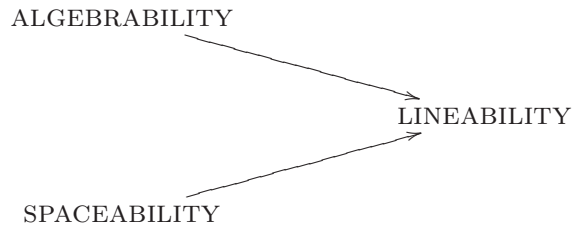
\[
\xymatrix{
     \textsc{algebrability} \ar[rrd] & &\\
     & & \textsc{lineability} \\
     \textsc{spaceability}  \ar[rru] & &
}
\]

Throughout this paper, and if $X$ denotes a sequence space, we shall denote by $Z(X)$ the subset of $X$ formed by sequences having only a finite number of zero coordinates. Therefore, the above question can be stated in terms of lineability and spaceability: 
\medskip
\begin{center}
{\em Is $Z(\ell_{\infty})$ spaceable in $\ell_{\infty}$?}
\end{center}
\medskip

Lately, these concepts of lineability and spaceability have proven to be quite fruitful and have attracted the interest of many mathematicians, among whom we have R. Aron, L. Bernal-Gonz\'alez, P. Enflo, G. Godefroy, V. Fonf, V. Gurariy, V. Kadets, or E. Teixeira  (see, e.g. \cites{bbfp,bcfp,BG_2006,BAMS,bernal2008,enflogurariyseoane2012,pt2009,Studia2012}). Question \ref{Q} has also appeared in several recent works (see, e.g., \cite{BAMS,enflogurariyseoane2012,munozpalmbergpuglisiseoane2008,garciaperezseoane2010}) and, for the last decade, there have been several attempts to partially answer it, although nothing conclusive in relation to the original problem has been obtained so far. 

This paper is arranged as follows. Section 2 shall focus on the algebrability (and, thus, lineability) of the set $Z(X)$ for $X \in \{c_0, \ell_p\}$, $p \in [1,\infty]$. Sections 3 and 4 will show that spaceability of $Z(X)$ is actually not possible for any of the previous Banach spaces whereas, in Section 5, we shall show that $V\setminus Z(V)$ is, actually, spaceable (and algebrable) for every infinite dimensional closed subspace (subalgebra) $V$ of $X$ (for $X \in \{c_0, \ell_p\}$, $p \in [1,\infty]$).

There are not many examples of (nontrivial) sets that are lineable and not spaceable. One of the first ones in this direction was due to B. Levine and D. Milman (1940, \cite{LM}) who showed that the subset of $\mathcal{C}[0,1]$ of all functions of bounded variation is not spaceable (it is obviously lineable, since it is a linear space itself). A more recent one was due to V. Gurariy (1966, \cite{gurariy1966}), who showed that the set of everywhere differentiable functions on $[0,1]$ (which is also an infinite dimensional linear space) is not spaceable in $\mathcal{C}([0,1])$. However, L. Bernal-Gonz\'alez (\cite{Bernal}, 2010) showed that $\mathcal{C}^\infty(]0,1[)$ is, actually, spaceable in $\mathcal{C}(]0,1[)$. 

Here, we shall provide (among other results) the definitive answer to Question \ref{Q}.  Namely, if $X$ stands for $c_0$, or $\ell_p$, with $p \in [1,\infty]$, we prove the following:

\begin{itemize}
\item[$i.-)$] $Z(X)$ is maximal algebrable and maximal lineable (\cite{bernal2010}), that is, $Z(X) \cup \{0\}$ contains an algebra with a system of generators of cardinality dim($X$) and a linear subspace of dimension dim($X$) (Proposition \ref{lin}).

\item[$ii.-)$] $Z(X)$ is not spaceable, that is, every closed subspace of $Z(X) \cup \{0\}$ must have finite dimension (Corollaries \ref{spaD} and \ref{spa}).

\item[$iii.-)$] $V\setminus Z(V)$ is maximal spaceable for every infinite dimensional closed subspace $V$ of $X$ (Theorem \ref{result}). 

\item[$iv.-)$] $V\setminus Z(V)$ is maximal algebrable for any infinitely generated closed subalgebra $V$ of $\ell_{p}$ (Theorem \ref{result2}).
\end{itemize}

In order to obtain the above results we shall make use of Functional Analysis techniques, Basic Sequences, Complemented Subspaces, and some classical Linear Algebra and Real Analysis approaches. From now on, if $Y$ is any sequence space and $y \in Y$, then $y(j)$ shall denote the $j-$th coordinate of $y$ with respect to the canonical basis $(e_j)_j$. Also, if $(m_{k})_{k\in\mathbb{N}}$ is a subsequence of $(n_{k})_{k\in\mathbb{N}}$, we shall write $(m_{k})_{k\in\mathbb{N}}\subset (n_{k})_{k\in\mathbb{N}}$.  If $V$ is a normed space and $(v_k)_{k\in\mathbb{N}}\subset V$, we denote by $\langle v_1,v_2,\ldots\rangle$ the linear span of $\{v_1,v_2,\dots\}$ and by $[v_1,v_2,\dots\ ]$ the closed linear span of $\{v_1,v_2,\dots\}$. If $W\subset V$, we denote $S_1(W)=\{w\in W,\ |w|=1\}$. The rest of the notation shall be rather usual.

\section{$Z(X)$ is maximal lineable for $X = c_0$ or $\ell_p$, $p \in [1,+\infty]$}

The following proposition would provide a positive answer to Question \ref{Q}, provided we remove the hypothesis of being closed in $\ell_\infty$.

\begin{proposition}\label{lin}
$Z(X)$ is maximal algebrable for $X = c_0$ or $\ell_p$, $p \in [1,+\infty]$.
\end{proposition}

\begin{proof}
For every real number $p\in \ ]0,1[$ denote $$x_p = \left( p^1, p^2, p^3, \ldots \right),$$
and let $V = \text{span}\{x_p : p \in\ ]0,1[\}.$  Notice that $V\subset X$, for $X = c_0$ or $\ell_p$, $p \in [1,+\infty]$. 

Next, take any $x \in V \setminus \{0\}$. We shall show that $x \in Z(X)$ . We can write $x$ as
$$x = \sum_{j=1}^{N} \lambda_j x_{p_j},$$
with $N\in \mathbb{N}$, $p_j \in\ ]0,1[$ for every $j \in \{1, 2, \ldots, N\}$, $p_N > p_{N-1} > \ldots > p_1$, and $(\lambda_j)_{j=1}^{N} \subset \mathbb{C}$.
Let us suppose that there exists an increasing sequence of positive integers $(m_k)_{k\in \mathbb{N}}$ such that $x(m_k) = 0$ for every $k \in \mathbb{N}$. Then, we have 
$$0 = \sum_{j=1}^{N} \lambda_j p_j^{m_k}$$
for every $k \in \mathbb{N}$.

Dividing the last identity by $p_N^{m_k}$, we obtain (for every $k \in \mathbb{N}$), 

\begin{equation} \label{eq1} 
0 = \sum_{j=1}^{N-1} \lambda_j \left(\frac{p_j}{p_N}\right)^{m_k}+\lambda_N.
\end{equation}

Now, since $0<\displaystyle \frac{p_j}{p_N}<1$ for every $j \in \{1, 2, \ldots, N-1\}$ and $\displaystyle\lim_{k\rightarrow\infty} m_k=\infty$, we have $\displaystyle\lim_{k\rightarrow\infty} \left(\frac{p_j}{p_N}\right)^{m_k}=0.$ Thus, $\lambda_N=0$ in equation \eqref{eq1}. By induction, we can easily obtain $\lambda_j = 0$ for every $j \in \{1, 2, \ldots, N\}$. This is a contradiction, since $x \neq 0$.

\noindent This argument also shows that $V$ is $\mathfrak{c}$-dimensional (where $\mathfrak{c}$ stands for the continuum) and, thus, $Z(X)$ is maximal lineable for $X = c_0$ or $\ell_p$, $p \in [1,+\infty]$.

Now let $x_p,x_q\in\{x_r,\ r\in\ ]0,1[\ \}$. Notice that the coordinatewise product of $x_p$ and $x_q$ is $x_{pq}\in \{x_r,\ r\in\ ]0,1[\ \}$. Therefore the algebra generated by $\{x_r,\ r\in\ ]0,1[\ \}$ is the subspace generated by $\{x_r,\ r\in\ ]0,1[\ \}$ which is $V$.

Consider any countable subset of $W\subset V$. The subalgebra generated by $W$ is a vector space generated by finite products of elements of $W$, but finite product of elements from a countable set still is countable. Therefore the subalgebra generated by $W$ has countable dimension, therefore $W$ cannot be a set of generators for the algebra $V$, since $\dim(V)$ is uncountable. Therefore any set of generators of $V$ is uncountable.
\end{proof}

The following result is a straightforward consequence of Proposition \ref{lin}.

\begin{corollary}
$Z(X)$ is maximal lineable for $X = c_0$ or $\ell_p$, $p \in [1,+\infty]$.
\end{corollary}

It this direction, we would like to recall that {\em similar} properties to the one considered here have also been studied within the framework of function spaces. For instance, let $X$ be any infinite dimensional closed subspace of $\mathcal{C}[0,1]$ and consider $Y$ to be the subset of functions in $X$ having infinitely many zeros in $[0,1]$. P. Enflo, V. Gurariy, and the second named author recently showed in \cite{enflogurariyseoane2012} that $Y$ is spaceable in $X$. In the next two sections we shall see that this is not the case for $Z(\ell_p)$, $p\in[1,\infty]$, and $Z(c_0)$.

\section{$Z(X)$ is not spaceable for $X =\ell_p$, $p\in [1,\infty[$}

We need a series of technical lemmas in order to achieve the main result of this section. We believe that these lemmas are of independent interest.

\begin{lemma}\label{lemmaA} 
Let $V$ be an infinite dimensional closed subspace of $\ell_{p}$, $p\in [1,\infty[$. Given $0<\epsilon<\frac{4}{33}$ there is an increasing sequence of natural numbers $(s_k)_{k\in\mathbb{N}}$ and a normalized basic sequence $(f_{k})_{k\in\mathbb{N}}\subset V$ such that
\begin{enumerate}
\item $f_{k}(s_j)=0$ for $1\leq j\leq k-1$.
\item $f_1(s_1)\neq 0$.
\item $|f_1(s_{k+1})|+\ldots+|f_{k}(s_{k+1})|<\frac{\epsilon}{2^{k+1}}|f_{k+1}(s_{k+1})|$  for every $k$\\ 
 $($thus $f_k(s_k)\neq 0$ for every $k\in\mathbb{N})$.
\item $(f_{k})_{k\in\mathbb{N}}$ has basis constant smaller than $\frac{8-2\epsilon}{4-9\epsilon}$. 
\item $[f_1,f_2,\dots]$ is complemented in $\ell_p$ with a projection $Q:\ell_p \rightarrow \ell_p$ of  norm $||Q||\leq\frac{8-2\epsilon}{4-33\epsilon}$.
\end{enumerate}
\end{lemma}

\begin{proof}  
Let $f_1\in V$ such that $|f_1|_p=1$. Let $N_1\in\mathbb{N}$ such that
\begin{enumerate}
\item $f_1(N_1)\neq 0$. 
\item  $|(f_1(n))_{n=N_1+1}^{\infty}|_p<\frac{\epsilon}{2^2}$. 
\end{enumerate}

Let $s_1=N_1$. Suppose we have defined $f_2,\ldots,f_t\in V$ and $$s_1=N_1<s_2<N_2<\ldots<s_t<N_t$$ such that 

\begin{enumerate}
\item $|f_k|_p=1$ for $1< k\leq t$
\item $f_k(n)=0$ for $1\leq n\leq N_{k-1}$ for every $1<k\leq t$ \\
$($Thus $f_k(s_j)=0$ for  $1\leq j\leq k-1$ since $ s_{k-1}<N_{k-1}).$
\item $|(|f_1(n)|+\ldots+|f_{k}(n)|)_{n=N_k+1}^{\infty}|_p<\frac{\epsilon}{2^{k+1}}$ for $1< k\leq t$
\item $|f_1(s_{k+1})|+\ldots+|f_{k}(s_{k+1})|<\frac{\epsilon}{2^{k+1}}|f_{k+1}(s_{k+1})|$ for $1< k\leq t-1$.\\
$($Thus $f_k(s_k)\neq 0$ for $1<k\leq t )$
\end{enumerate}

Since $V$ is an infinite dimensional closed subspace of $\ell_{p}$, there exists $f_{t+1}\in V$ such that $|f_{t+1}|_p=1$ and $f_{t+1}(1)=\ldots=f_{t+1}(N_{t})=0$.\\

Now, if there is not $n>N_t$ such that $|f_1(n)|+\ldots+|f_{t}(n)|<\frac{\epsilon}{2^{t+1}}|f_{t+1}(n)|$ then
$$\frac{\epsilon}{2^{t+1}}>|(|f_1(n)|+\ldots+|f_{t}(n)|)_{n=N_t+1}^{\infty}|_p \geq \frac{\epsilon}{2^{t+1}}|f_{t+1}|_p=\frac{\epsilon}{2^{t+1}},$$

which is absurd. Therefore exist $s_{t+1}>N_t$ such that 
$$|f_1(s_{t+1})|+\ldots+|f_{t}(s_{t+1})|<\frac{\epsilon}{2^{t+1}}|f_{t+1}(s_{t+1})|.$$

Next, since $(|f_1(n)|+\ldots+|f_{t+1}(n)|)_{n\in\mathbb{N}}\in\ell_p$ then exist $N_{t+1}>s_{t+1}$ such that 
 $$|(|f_1(n)|+\ldots+|f_{t+1}(n)|)_{n=N_{t+1}+1}^{\infty}|_p<\frac{\epsilon}{2^{t+2}}.$$

The induction to construct $(f_{k})_{k\in\mathbb{N}}$ enjoying the four above properties is now complete. Now, in order to show that $(f_k)_{k\in\mathbb{N}}$ is a basic sequence, let us define 
$$\widetilde{f_1}(n)=\left\lbrace
\begin{array}{c}
f_{1}(n),\text{ if } 1\leq n\leq N_{1}\\ 
0,\hspace{0.3cm} \text{ otherwise  }
\end{array}\right.\hspace{1cm}
\widetilde{f_k}(n)=\left\lbrace
\begin{array}{c}
f_{k}(n),\text{ if } N_{k-1}< n\leq N_{k}\\ 
0,\hspace{0.3cm} \text{ otherwise  }
\end{array}\right. $$

Notice that $\widetilde{f_k}\neq 0$, since $N_{k-1}<s_k<N_k$ and $\widetilde{f_k}(s_k)=f_k(s_k)\neq 0$. Note also that $\widetilde{f_k}$  is a block basis of the canonical basis of $\ell_p$. 

Since $$|(|f_1(n)|+\ldots+|f_k(n)|)_{n=N_{k}+1}^{\infty}|_p<\frac{\epsilon}{2^{k+1}},$$ then $$|(f_k(n))_{n=N_{k}+1}^{\infty}|_p<\frac{\epsilon}{2^{k+1}}.$$ 

Now since $f_k(n)=0$ for $1\leq n\leq N_{k-1}$ we obtain $$ 1-\frac{\epsilon}{2^{k+1}}\leq |\widetilde{f_k}|_p\leq 1\text{ and }|f_k-\widetilde{f_k}|_p<\frac{\epsilon}{2^{k+1}}$$ for $k\in\mathbb{N}$. In particular, $\frac{4-\epsilon}{4} =1-\frac{\epsilon}{4}\leq |\widetilde{f_k}|_p\leq 1$ for every $k\in\mathbb{N}$. Let $g_k=\frac{\widetilde{f_k}}{|\widetilde{f_k}|_p}$ for every $k$. Notice that $(g_k)_{k\in\mathbb{N}}$ is a normalized block basis of the canonical basis of $\ell_p$. So $|\sum_{k=1}^{\infty}a_k g_k|_p= |(a_k)_{k\in\mathbb{N}}|_p$ and $(g_k)_{k\in\mathbb{N}}$ has basis constant $K=1$. Let $\{\sigma_k,\ k\in\mathbb{N}\}$ be the following partition of $\mathbb{N}$:
$$\sigma_1=\{1,\ldots,N_1\}\text{ and }\sigma_k=\{N_{k-1}+1, \ldots, N_{k}\}.$$

Next, let $E_k=\{f\in\ell_p,\ f(i)=0,\text{ for }i\notin\sigma_k\}$. Thus, $g_k\in E_k$ and by \cite[Theorem 30.18]{jameson} the closed subspace $[g_1,g_2,\ldots]$ is complemented in $\ell_p$ with a projection $P:\ell_p \rightarrow \ell_p$ of  norm $1$.

Let us now prove that $(f_k)_{k\in\mathbb{N}}$ is equivalent to $(g_k)_{k\in\mathbb{N}}$ and $[f_1,f_2,\ldots]$ is also complemented in $\ell_p$. Indeed, 
\begin{align*}
|f_k-g_k|_p & = |f_k-\frac{\widetilde{f_k}}{|\widetilde{f_k}_p|}|_p\leq |f_k-\frac{f_k}{|\widetilde{f_k}|_p}|_p+|\frac{f_k}{|\widetilde{f_k}|_p}-\frac{\widetilde{f_k}}{|\widetilde{f_k}|_p}|_p\\
& \leq \frac{1-|\widetilde{f_k}|_p}{|\widetilde{f_k}|_p}+\frac{1}{|\widetilde{f_k}|_p}\frac{\epsilon}{2^{k+1}}\leq\frac{4}{4-\epsilon}\left(1-|\widetilde{f_k}|_p+\frac{\epsilon}{2^{k+1}}\right)\\
& \leq \frac{4}{4-\epsilon}\left(\frac{2\epsilon}{2^{k+1}}\right).
\end{align*}

Thus, $(g_k)_{k\in\mathbb{N}}$ is a normalized basic sequence such that $[g_1,g_2,\ldots]$ is complemented in $\ell_p$ with a projection $P:\ell_p \rightarrow \ell_p$ of  norm $1$  and  $$\delta=\displaystyle\sum_{k=1}^{\infty}|f_k-g_k|_p\leq\sum_{k=1}^{\infty}\frac{4}{4-\epsilon}\frac{\epsilon}{2^{k}}=\frac{4\epsilon}{4-\epsilon}.$$ 

Since $0<\epsilon<\frac{4}{33}$, we obtain $8K\delta||P||=8\delta\leq 8 \frac{4\epsilon}{4-\epsilon}<1$. By the {\em principle of small perturbation} (\cite[Theorem 4.5]{carothers}) the sequence $(f_k)_{k\in\mathbb{N}}$ is equivalent to $(g_k)_{k\in\mathbb{N}}$ and $[f_1,f_2,\ldots]$ is also complemented in $\ell_p$.

Finally, let us compute an upper bound for the basis constant of $(f_k)_{k\in\mathbb{N}}$ and for the norm of the projection $Q:\ell_p\rightarrow\ell_p$ onto $[f_1,f_2,\ldots]$.

First, the linear transformation $T(\sum_{k=1}^{\infty}a_k g_k)=\sum_{k=1}^{\infty}a_kf_k$ is an invertible continuous linear transformation from the closed span of $( g_k)_{k\in\mathbb{N}}$ to the closed span of $(f_k)_{k\in\mathbb{N}}$. 

In \cite[Theorem 4.5]{carothers} it is proved that $||T||\leq (1+2K\delta)\leq (1+8\delta)\leq 2$ and  $||T^{-1}||\leq(1-2K\delta)^{-1}$. Let $P_n(\sum_{k=1}^\infty a_k g_k)=\sum_{k=1}^n a_k g_k$. Notice that $||P_n||=1$.

Thus, for $n\leq m$, 
\begin{align*}
|\sum_{k=1}^na_kf_k|_p & =|T\circ P_n\circ T^{-1}(\sum_{k=1}^ m a_kf_k)|_p\leq ||T||\ ||P_n||\ ||T^{-1}||  |\sum_{k=1}^ m a_kf_k|_p\\
& \leq \frac{2}{1-2K\delta} |\sum_{k=1}^ m a_kf_k|_p.
\end{align*}
Then, the basis constant of $(f_k)_{k\in\mathbb{N}}$ is smaller than $\frac{2}{1-2K\delta}\leq\frac{8-2\epsilon}{4-9\epsilon}$, since $K=1$ and $\delta\leq\frac{4\epsilon}{4-\epsilon}$. Again, using \cite[Theorem 4.5]{carothers}, the linear transformation $$Id-(T\circ P):[f_1,f_2,\ldots]\rightarrow [f_1,f_2,\ldots]$$
is invertible and having norm smaller than $8K\delta ||P||=8\delta<1$.

Therefore, there exists an inverse for $S=T\circ P:[f_1,f_2,\ldots]\rightarrow [f_1,f_2,\ldots]$ with norm $||S^{-1}||\leq \frac{1}{1-8\delta}$. Now $Q=S^{-1}\circ (T\circ P):\ell_p\rightarrow\ell_p$ is a projection onto $[f_1,f_2,\ldots]$ with norm $||Q||\leq||S^{-1}||\ ||T||\ ||P||=\frac{1}{1-8\delta}\times 2\times 1\leq \frac{8-2\epsilon}{4-33\epsilon}$, since $\delta\leq\frac{4\epsilon}{4-\epsilon}$.
\end{proof}

\begin{lemma}\label{lemmaB} 
Let $V$ be an infinite dimensional closed subspace of $\ell_{p}$, $p\in [1,\infty[$. There exist an increasing sequence of natural numbers $(s_k)_{k\in\mathbb{N}}$ and a basic sequence $(l_{s_k})_{k\in\mathbb{N}}\subset V$ such that
\begin{enumerate}
\item $l_{s_k}(s_k)\neq 0$
\item $l_{s_k}(s_j)=0$ for $k\neq j$
\item $[l_{s_1},l_{s_2},\ldots]$ is complemented in $\ell_p$.
\end{enumerate}
\end{lemma}
\begin{proof}
Let $0<\epsilon<\frac{1}{512}$ then $4-9\epsilon>1$,  $4-33\epsilon>1$ and $$8\epsilon\left(\frac{8-2\epsilon}{4-9\epsilon}\right)\left(\frac{8-2\epsilon}{4-33\epsilon}\right)<512\epsilon<1.$$

Let $(s_k)_{k\in\mathbb{N}}$ and $(f_{k})_{k\in\mathbb{N}}$ be as in lemma \ref{lemmaA}, using this $\epsilon$.

Define $l_{0,k}=f_k$. Notice that $l_{0,k}(s_k)=f_k(s_k)\neq 0$ and $l_{0,k}(s_j)=0$ for $s_j\in\{s_1,\ldots, s_k\}\setminus\{s_k\}$. Define 
$$l_{1,k}=l_{0,k}-\frac{l_{0,k}(s_{k+1})}{f_{k+1}(s_{k+1})}f_{k+1}.$$

Notice that 
\begin{enumerate}
\item $l_{1,k}(s_j)=0$ for $s_j\in\{s_1,\ldots, s_k,s_{k+1}\}\setminus\{s_k\}$.
\item $l_{1,k}(s_k)=f_k(s_k)\neq 0$.
\item Since $|l_{0,k}(s_{k+1})|=|f_{k}(s_{k+1})|<\frac{\epsilon}{2^{k+1}}|f_{k+1}(s_{k+1})|$
thus $\frac{|f_{k}(s_{k+1})|}{|f_{k+1}(s_{k+1})|}<\frac{\epsilon}{2^{k+1}}<1$\\
and $|l_{1,k}(n)|\leq |f_k(n)|+|f_{k+1}(n)| $ for every $n\in\mathbb{N}$.
\item $|l_{1,k}-l_{0,k}|_p<\frac{\epsilon}{2^{k+1}}|f_{k+1}|_p=\frac{\epsilon}{2^{k+1}}.$\\
\end{enumerate}

Supposed we have already defined $l_{0,k},\ldots, l_{t,k}$ such that
\begin{enumerate}
\item $l_{i,k}(s_j)=0$ for $s_j\in\{s_1,\ldots,s_{k+i}\}\setminus\{s_k\}$, for $1\leq i\leq t$
\item $l_{i,k}(s_k)=f_k(s_k)\neq 0$, for $1\leq i\leq t$
\item $|l_{i,k}(n)|\leq |f_k(n)|+\ldots+|f_{k+i}(n)|$, for every $n\in\mathbb{N}$ and for $1\leq i\leq t$
\item $|l_{i,k}-l_{i-1,k}|_p<\frac{\epsilon}{2^{k+i}}$, for $1\leq i\leq t$.\\
\end{enumerate}

Define $l_{t+1,k}=l_{t,k}-\frac{l_{t,k}(s_{k+t+1})}{f_{k+t+1(s_{k+t+1})}}f_{k+t+1}$. Since $f_{k+t+1}(s_j)=0$ for $1\leq j\leq k+t$ then $l_{t+1,k}(s_j)=l_{t,k}(s_j)$for $1\leq j\leq k+t$. Since $l_{t+1,k}(s_{k+t+1})=0$ then 

\begin{enumerate}
\item $l_{t+1,k}(s_j)=0$ for $s_j\in\{s_1,\ldots, s_{k+t+1}\}\setminus\{s_k\}$
\item $l_{t+1,k}(s_k)=l_{t,k}(s_k)=f_{k}(s_k)\neq 0$
\item $|l_{t,k}(s_{k+t+1})|\leq |f_k(s_{k+t+1})|+\ldots+|f_{k+t}(s_{k+t+1})|$\\
 $\leq |f_1(s_{k+t+1})|+\ldots+|f_{k+t}(s_{k+t+1})|<\frac{\epsilon}{2^{k+t+1}}|f_{k+t+1}(s_{k+t+1})|.$ \\
 Therefore $\frac{|l_{t,k}(s_{k+t+1})|}{|f_{k+t+1}(s_{k+t+1})|}<\frac{\epsilon}{2^{k+t+1}}<1$ and\\
 $|l_{t+1,k}(n)|\leq |l_{t,k}(n)|+|f_{k+t+1}(n)|\leq |f_k(n)|+\ldots+|f_{k+t+1}(n)|$\\
  for every $n\in\mathbb{N}$.
\item $|l_{t+1,k}-l_{t,k}|_p<\frac{\epsilon}{2^{k+t+1}}|f_{k+t+1}|_p=\frac{\epsilon}{2^{k+t+1}}.$
\end{enumerate}

The induction to construct $(l_{t,k})_{t=0}^{\infty}$ for each $k\in\mathbb{N}$ is completed. Next, let $t>m$ and notice that  $$|l_{t,k}-l_{m,k}|_p=|l_{t,k}-l_{t-1,k}|_p+\ldots+|l_{m+1,k}-l_{m,k}|_p\leq \frac{\epsilon}{2^{k+t}}+\ldots+\frac{\epsilon}{2^{k+m+1}}\leq \frac{\epsilon}{2^{k+m}}.$$
Therefore $(l_{t,k})_{t=0}^{\infty}$ is a cauchy sequence in $V$, for each $k$. Let $\displaystyle\lim_{t\rightarrow\infty}l_{t,k}=l_k\in V$. Now notice that 

\begin{enumerate}
\item Since for every $t$, we have $l_{t,k}(s_k)=f_k(s_k)\neq 0$  then \\ $\displaystyle l_k(s_k)=\lim_{t\rightarrow\infty}l_{t,k}(s_k)=f_k(s_k)\neq 0$
\item Since for $t>j$ and $j\neq k$, we have  $l_{t,k}(s_j)=0$ then\\
 $\displaystyle l_k(s_j)=\lim_{t\rightarrow\infty} l_{t,k}(s_j)=0$.
\item Since $|l_{t,k}-l_{0,k}|_p\leq\frac{\epsilon}{2^k}$ then
$|l_k-f_k|_p=\displaystyle \lim_{t\rightarrow\infty} |l_{t,k}-l_{0,k}|_p\leq \frac{\epsilon}{2^k}.$
\end{enumerate}

Thus, $(f_k)_{k\in\mathbb{N}}$ is a normalized basic sequence with basis constant $K\leq \frac{8-2\epsilon}{4-9\epsilon}$  such that $[f_1,f_2,\ldots]$ is complemented in $\ell_p$ with a projection $P:\ell_p\rightarrow\ell_p$ with norm $||P||\leq \frac{8-2\epsilon}{4-33\epsilon}$  and $$\displaystyle\delta=\sum_{k=1}^{\infty}|l_k-f_k|_p\leq \sum_{k=1}^{\infty} \frac{\epsilon}{2^k}=\epsilon.$$

Finally $8K\delta ||P||\leq 8\epsilon\left(\frac{8-2\epsilon}{4-9\epsilon}\right)\left(\frac{8-2\epsilon}{4-33\epsilon}\right)<512\epsilon<1.$ By the principle of small pertubation \cite[Theorem 4.5]{carothers} the sequence $(l_k)_{k\in\mathbb{N}}$ is equivalent to $(f_k)_{k\in\mathbb{N}}$ and $[l_1,l_2,\ldots]$ is complemented in $\ell_p$. Finally define $l_{s_k}=l_k$ for $k\in\mathbb{N}$.
\end{proof}

\begin{proposition}\label{propC}
Let $V$ be an infinite dimensional closed subspace of $\ell_{p}$, $p\in\ [1,\infty[$. There exists $0 \neq h \in V \setminus Z(V)$.
\end{proposition}
\begin{proof}
Consider any $l_{k}$ from Lemma \ref{lemmaB}. Notice that any $l_{k}\in V\setminus Z(V)$.
\end{proof}

\begin{corollary}\label{spaD}
$Z(\ell_p)$ is not spaceable in $\ell_{p}$, for $p\in\ [1,\infty[$.
\end{corollary}

\begin{corollary} 
Let $V$ be an infinite dimensional closed subspace of $\ell_p$. Then $V \setminus Z(V)$ is dense in $V$.
\end{corollary}

\begin{proof}
Let $0\neq f\in V $. Define $f_1=\frac{f}{|f|_p}$. We can start the proof of lemma \ref{lemmaA} using this $f_1$. Consider the proof of lemma \ref{lemmaB}. For  a sufficiently small $\epsilon$ $($independent of $|f|_p)$, we found a $l_1\in V\setminus Z(V)$ such that $|f_1-l_1|_p<\frac{\epsilon}{2^1}$ then $$|f-|f|_p\ l_1|<\frac{|f|_p\epsilon}{2}.$$ Now $0\neq |f|_pl_1\in V\setminus Z(V)$.
\end{proof}

\section{$Z(X)$ is not spaceable for $X = c_0$ or $\ell_\infty$}

This section shall provide the definitive answer to Question \ref{Q} by showing that $Z(\ell_\infty)$ is not spaceable. In other words, $\ell_\infty$ does not contain infinite dimensional Banach subspaces every nonzero element of which has only a finite number of zero coordinates. In order to achieve this we shall need to obtain a sequence $(l_{s_k})_{k\in\mathbb{N}}$ similar to that from Lemma \ref{lemmaB} (see Lemma \ref{lemmamain2}). Despite losing the hypothesis of the closed span of $(l_{s_k})_{k\in\mathbb{N}}$ being complemented, we gain the property $l_{s_k}(s_k)=1$, obtaining still a basic sequence.

\begin{definition} Let $V$ be an infinite dimensional closed subspace of $\ell_{\infty}$. Let $s\in\mathbb{N}$ and define
$$V_{s}=\left\{f\in V,\ f\neq 0,\  |f(s)|\geq\dfrac{|f|_{\infty}}{2}\right\}.$$
\end{definition}

\begin{lemma}\label{lemma0} Let $V$ be an infinite dimensional closed subspace of $\ell_{\infty}$.\\
 For every  $K\subset V$, $K\neq \{0\}$, exist $s\in\mathbb{N}$ such that $$V_{s}\cap K\neq \emptyset.$$
\end{lemma}
\begin{proof}
  Let $f\in K$, $f\neq 0$. Since $|f|_{\infty}=\sup_{k\in\mathbb{N}}|f(k)|$ there is $s\in\mathbb{N}$ such that $|f(s)|\geq \frac{|f|_{\infty}}{2}$. So $f\in V_{s}\cap K$.
\end{proof}

\begin{lemma}\label{lemma1} 
Let $V$ be an infinite dimensional closed subspace of $\ell_{\infty}$. There exists an increasing sequence of natural numbers $(n_k)_{k\in{\mathbb{N}}}$ and a basic sequence $(f_{n_k})_{k\in\mathbb{N}}\subset V$ with:
\begin{enumerate}
\item $f_{n_k}(n_k)=1$,
\item $f_{n_j}(n_i)=0$ for $j>i$, and
\item $|f_{n_k}|_{\infty}\leq 2$ for every $k\in\mathbb{N}$.
\end{enumerate}
\end{lemma}

\begin{proof} 
This proof is a variation of Mazur's lemma (\cite[Proposition 4.1]{carothers}).

 Let $\epsilon_1=1$ and $\epsilon_i\in\ ]0,1[$ such that $\displaystyle\prod_{i=1}^{\infty}(1+\epsilon_i)<\infty$.

By Lemma \ref{lemma0} exist $s\in\mathbb{N}$ such that $V_{s}=V_{s}\cap  V\neq\emptyset$. 

Let $n_1=\min\{s\in\mathbb{N}, V_{s}\neq\emptyset\}$ and let $f_1\in V_{n_1}$. Define $$f_{n_1}=\frac{f_1}{f(n_1)}.$$ 

Notice that 
$$f_{n_1}(n_1)=1 \, \text{ and } \,  1\leq |f_{n_1}|_{\infty}=\frac{|f_1|_{\infty}}{|f_1(n_1)|}\leq 2.$$ 

Consider the projection $\pi_{n_1}:V\rightarrow\mathbb{C}$, $\pi_{n_1}(f)=f(n_1)$.  Let $W_1=\ker(\pi_{n_1})$. Since $\text{codim}(W_1)$ in $V$ is finite then $\dim(W_1)=\infty$, by lemma \ref{lemma0}, exist $s\in\mathbb{N}$ such that $V_{s}\cap W_1\neq\emptyset$. 

Let $n_2=\min\{s\in\mathbb{N}, V_{s}\cap W_1\neq\emptyset\}$.
Since  $V_{s}\supset V_{s}\cap W_1$ then $n_2\geq n_1$.

Now for every $f\in W_1,\ f(n_1)=0$ then $V_{n_1}\cap W_1=\emptyset$ then $n_2>n_1$. Next, let $f_2 \in V_{n_2}\cap W_1$ and define $$f_{n_2}=\frac{f_2}{f_2(n_2)}.$$ 

Notice that  $$f_{n_2}(n_2)=1, \,  f_{n_2}(n_1)=0, \, \text{ and } \,  1\leq |f_{n_2}|_{\infty}=\frac{|f_2|_{\infty}}{|f_2(n_2)|}\leq 2.$$ 

Next, $|a_1f_{n_1}+a_2f_{n_2}|_{\infty}\geq |\pi_{n_1}(a_1f_{n_1}+a_2f_{n_2})|=|a_1|$ and $1+\epsilon_1=2\geq |f_{n_1}|_{\infty}$, so

 $$|a_1f_{n_1}+a_2f_{n_2}|_{\infty}(1+\epsilon_1)\geq |a_1||f_{n_1}|_{\infty}=|a_1f_{n_1}|_{\infty}.$$

Consider now the compact set $S_1(\langle f_{n_1},f_{n_2}\rangle)$ and let $\{y_1,\ldots,y_k\}\subset S_1(\langle f_{n_1},f_{n_2}\rangle)$ be such that if $y\in S_1(\langle f_{n_1}, f_{n_2}\rangle)$ then exist $y_i$ such that $|y-y_i|_{\infty}<\frac{\epsilon_2}{2}$. Consider $\{\phi_1,\dots,\phi_k\}\subset S_1(V^{*})$ such that $\phi_{i}(y_i)=1$.

Take $\pi_{n_2}:V\rightarrow \mathbb{C}$, $T_{n_2}(f)=f(n_2)$. Let $$\displaystyle W_2=\bigcap_{i=1}^k\ker(\phi_i)\cap\ker(\pi_{n_2})\cap W_1.$$ 

Since $\text{codim}(\ker_{\phi_i}),$ $\text{codim}(\ker_{\pi_{n_2}})$, and $\text{codim}(W_1)$ are finite in $V$ then $\text{codim}(W_2)$ is finite and $\dim(W_2)=\infty$. By  lemma \ref{lemma0} exist $s\in\mathbb{N}$ such that $V_{s}\cap W_2\neq\emptyset$.

Let $n_3=\min\{s\in\mathbb{N},\ V_{s}\cap W_2\neq\emptyset \}$. Since $V_{s}\cap W_1\supset V_{s}\cap W_2$ then $n_3\geq n_2$.

Now, for every $f\in W_2,\ f(n_2)=0$ then $V_{n_2}\cap W_2=\emptyset$ then $n_3>n_2$ . Next, let $f_3 \in V_{n_3}\cap W_2$ and define $$f_{n_3}=\frac{f_3}{f_3(n_3)}.$$ 

Notice that 
$$f_{n_3}(n_3)=1, \,  f_{n_3}(n_2)=f_{n_3}(n_1)=0, \, \text{ and } \,   1\leq |f_{n_3}|_{\infty}=\frac{|f_3|_{\infty}}{|f_3(n_3)|}\leq 2.$$  

Now, let $y \in S_1(\langle f_{n_1},f_{n_2} \rangle)$. Notice that 
\begin{align*}
|y+\lambda f_{n_3}|_{\infty} &\geq |y_i+\lambda f_{n_3}|_{\infty}-|y_i-y|_{\infty}\\
& \geq |y_i+\lambda f_{n_3}|_{\infty}-\frac{\epsilon_2}{2} \quad (\text{for some } i\in\{1,\ldots,k\})\\
& \geq \phi_i(y_i+\lambda f_{n_3})-\frac{\epsilon_2}{2}\\
& \geq \phi_i(y_i)-\frac{\epsilon_2}{2}\\
& \geq 1-\frac{\epsilon_2}{2}\geq \frac{1}{1+\epsilon_2}.
\end{align*}

Thus, for every $y\in S_1(\langle f_{n_1},f_{n_2}\rangle)$ and any $\lambda\in\mathbb{C}$ we have $$|y+\lambda f_{n_3}|_{\infty}(1+\epsilon_2)\geq |y|_{\infty}.$$

Then $$|a_1f_{n_1}+a_{2}f_{n_2}+a_3f_{n_3}|_{\infty}(1+\epsilon_2)\geq |a_1f_{n_1}+a_{2}f_{n_2}|_{\infty}$$ 
for every $a_1,a_2,a_3$ in $\mathbb{C}$. We can repeat the procedure to build $f_{n_4},f_{n_5},\ldots$ satisfying 
$$|a_1f_{n_1}+\ldots+a_kf_{n_m}|_{\infty}(1+\epsilon_{m-1})\ldots(1+\epsilon_k)\geq |a_1f_{n_1}+\ldots+a_kf_{n_k}|_{\infty}$$ 
for every $a_1,\ldots,a_m\in\mathbb{C}$ and $m\geq k$ and by Banach's criterion $(f_{n_k})_{k\in\mathbb{N}}\subset V$ is a basic sequence. Note that $(f_{n_k})_{k\in\mathbb{N}}$ satisfies the hypothesis.
\end{proof}


\begin{lemma}\label{lemma2} 
Let $g_1,g_2\in \ell_{\infty}$ and let $(m_{k})_{k\in\mathbb{N}}$ be an increasing sequence of natural numbers. There exists $(m_{k}^1)_{k\in\mathbb{N}}\subset (m_{k})_{k\in\mathbb{N}}$ such that
\begin{enumerate}
\item $\displaystyle \lim_{k\rightarrow\infty} g_1(m_k^1)=L_1$,
\item $\displaystyle \lim_{k\rightarrow\infty} g_2(m_k^1)=L_2$, and
\item $m_2^1>m_1^1>m_2>m_1$.
\end{enumerate}
\end{lemma}

\begin{proof}
The sequence $(g_1(m_k))_{k\in\mathbb{N}}$ is bounded since $g_1\in \ell_\infty$, therefore there is a subsequence $(m^0_k)_{k\in{\mathbb{N}}}\subset (m_k)_{k\in{\mathbb{N}}}$ and $L_1\in\mathbb{C}$ such that $\displaystyle\lim_{k\rightarrow\infty} g_1(m^0_k)=L_1$. Next, and by same reasoning, there is a subsequence $(m^1_k)_{k\in{\mathbb{N}}}\subset (m^0_k)_{k\in{\mathbb{N}}}$ and $L_2$ such that $\displaystyle\lim_{k\rightarrow\infty}g_2(m^1_k)=L_2$. Therefore  $\displaystyle\lim_{k\rightarrow\infty} g_1(m^1_k)=L_1$ and $\displaystyle\lim_{k\rightarrow\infty}g_2(m^1_k)=L_2$. Removing, if necessary, the first two terms in the sequence $(m^1_k)_{k\in{\mathbb{N}}}$ we may assume that $m_2^1>m_1^1>m_2>m_1$.
\end{proof}

\begin{lemma} \label{lemmamain} Let $V$ be an infinite dimensional closed subspace of $\ell_{\infty}$ and let  $(n_{k})_{k\in\mathbb{N}}$  and $f_{n_k}$ be as in Lemma \ref{lemma1}. For every $(m_{k})_{k\in\mathbb{N}}\subset (n_{k})_{k\in\mathbb{N}}$ there exist $(t_{k})_{k\in\mathbb{N}}\subset (m_{k})_{k\in\mathbb{N}}$ and basic sequence  $(h_{t_k})_{k\in\mathbb{N}}\subset V$ satisfying
\begin{itemize}
\item[$a)$] $h_{t_k}(t_s)=0$ for $s<k$,
\item[$b)$] $h_{t_k}(t_k)=1$,
\item[$c)$] $|h_{t_k}|_{\infty}\leq 8$, and
\item[$d)$] $\displaystyle\lim_{s\rightarrow\infty}h_{t_k}(t_s)=0.$
\end{itemize}
\end{lemma}

\begin{proof}
First of all, let us define $g_1=f_{m_1}-f_{m_1}(m_2)f_{m_2}$ and $g_2=f_{m_2}$. Notice that $g_1(m_1)=1$, $g_1(m_2)=0$, $g_2(m_1)=0$ and $g_2(m_2)=1$. Now, by Lemma \ref{lemma2}, there exists $(m_{k}^1)_{k\in\mathbb{N}}\subset (m_{k})_{k\in\mathbb{N}}$ such that
$$\lim_{k\rightarrow\infty} g_1(m_k^1)=L_1, \quad \lim_{k\rightarrow\infty} g_2(m_k^1)=L_2, \, \text{ and } \, m_2^1>m_1^1>m_2>m_1.$$
We now have the following possibilities.
\begin{enumerate}
\item If $L_1=0$, let $h_1=g_1$. Notice that, since $|f_{m_i}|_{\infty}\leq 2$  ($1\leq i\leq 2$) we have $|h_1|_{\infty}\leq 6.$ Notice also that $h_{1}(m_1)=1$.
\item If $L_1\neq 0$ and $L_2=0$, let $h_1=g_2$. We have $|h_1|_{\infty}\leq 2$ and $h_1(m_2)=1$.
\item If $L_1\neq 0$, $L_2 \neq 0$ and $|L_1|\leq |L_2|$, define $h_1=g_1-\dfrac{L_1}{L_2}g_2$. Notice that $|h_1|_{\infty}\leq |g_1|_{\infty}+\dfrac{|L_1|}{|L_2|}|g_2|_{\infty}\leq 8$. Also, $h_{1}(m_1)=1$.
\item Finally, if $L_1\neq 0$, $L_2 \neq 0$ and $|L_2|\leq |L_1|$, let $h_1=g_2-\dfrac{L_2}{L_1}g_1$, having now that $|h_1|_{\infty}\leq |g_2|_{\infty}+\dfrac{|L_2|}{|L_1|}|g_1|_{\infty}\leq 8$. Also, note that $h_{1}(m_2)=1$.
\end{enumerate}
Next, if $h_1(m_1)=1$, define $t_1=m_1$ and, if $h_1(m_1)\neq 1$, then $h_1(m_2)=1$ and we let $t_1=m_2$. In any case, note that $\displaystyle\lim_{k\rightarrow\infty}h_1(m_k^1)=0.$
Let us now suppose that, by induction, we have already defined
\begin{enumerate}
\item $(m_k^i)_{k\in\mathbb{N}}\subset \cdots \subset (m_k^1)_{k\in\mathbb{N}}\subset (m_k)_{k\in\mathbb{N}}$ with 
$$m_{2}^i>m_1^i>m_{2}^{i-1}>m_1^{i-1}>\ldots>m_2^1>m_1^1>m_2>m_1,$$
\item $t_{1}=m_{1}$ or $m_2$ and $t_{j}=m_{1}^{j-1}$ or $m_2^{j-1}$, $2\leq j\leq i$. 
\item $h_j\in V$, $1\leq j\leq i$, verifying items $a), b)$ and $c)$ of this lemma, and 
\item $\displaystyle \lim_{k\rightarrow\infty} h_j(m_k^j)=0$, $1\leq j\leq i$.
\end{enumerate}

Next, repeat the construction of $h_1$ in order to obtain $h_{i+1}$ by means of $f_{m_1^i},f_{m_2^i}$ instead of $f_{m_1},f_{m_2}$.

Using the sequence $(m_k^i)_{k\in\mathbb{N}}$, instead of $(m_k)_{k\in\mathbb{N}}$ in the previous construction, we obtain $(m_k^{i+1})_{k\in\mathbb{N}}\subset (m_k^i)_{k\in\mathbb{N}}$ such that 
$$m_{2}^{i+1}>m_1^{i+1}>m_{2}^i>m_1^i  \quad \text{ and} \quad \displaystyle\lim_{k\rightarrow\infty}h_{i+1}(m_k^{i+1})=0.$$

Define now $t_{i+1}=m_{1}^{i}$ or $m_2^{i}$, depending on whether $h_{i+1}(m_1^{i})=1$ or $h_{i+1}(m_2^{i})=1$, as we previously did for $t_1$. Therefore we have $h_{i+1}(t_{i+1})=1$. Next, since $h_{i+1}$ is a linear combination of $f_{m_1^i},f_{m_2^i}$, and

$$m_{2}^i>m_1^i>m_{2}^{i-1}>m_1^{i-1}>\ldots>m_2^1>m_1^1>m_2>m_1,$$ we obtain that $h_{i+1}(m_{1})=h_{i+1}(m_{2})=h_{i+1}(m_{1}^{j-1})=h_{i+1}(m_{2}^{j-1})=0$ (for $2\leq j\leq i$), but $t_1=m_{1}$ or $m_2$, $t_{j}=m_{1}^{j-1}$ or $m_2^{j-1}$ (for $2\leq j\leq i$),  which implies that $h_{i+1}(t_j)=0$ for $1\leq j\leq i$.

Finally, notice that $(t_s)_{s=i+1}^{\infty}\subset (m_k^i)_{k\in\mathbb{N}}$, thus $\displaystyle \lim_{s\rightarrow\infty}h_{i}(t_s)=0$ (for every $i\in\mathbb{N}$).

Notice that $(f_{m_k})_{k\in\mathbb{N}}$ is a basic sequence as subsequence of the basic sequence $(f_{n_k})_{k\in\mathbb{N}}$.\\
Notice also that  $h_{k}$ is a linear combination of $f_{m_1^{k-1}}$ and $f_{m_1^{k-1}}$, $h_1$ is a linear combination of $f_{m_1}$ and $f_{m_2}$ and 
$m_{2}^{k-1}>m_1^{k-1}>\ldots>m_2^1>m_1^1>m_2>m_1$ for every $k$. Therefore $(h_{k})_{k\in\mathbb{N}}$
is a block sequence of the basic sequence $(f_{m_k})_{k\in\mathbb{N}}$. Therefore $(h_k)_{k\in\mathbb{N}}$ is also a basic sequence. Finally, let $h_{t_k}=h_k$.
\end{proof}

\begin{lemma}
\label{lemmamain2} Let $V$ be an infinite dimensional closed subspace of $\ell_{\infty}$ and let  $(n_{k})_{k\in\mathbb{N}}$ be as in Lemma \ref{lemma1}. For every $(m_{k})_{k\in\mathbb{N}}\subset (n_{k})_{k\in\mathbb{N}}$ there exist $(s_{k})_{k\in\mathbb{N}}\subset (m_{k})_{k\in\mathbb{N}}$ and a basic sequence $(l_{s_k})_{k\in\mathbb{N}}\in V$, satisfying
\begin{itemize}
\item[$a)$] $l_{s_k}(s_k)=1$,
\item[$b)$] $l_{s_k}(s_j)=0$, for $j\neq k$.
\item[$c)$] $|l_{s_k}|_{\infty}\leq 9$, for every $k\in\mathbb{N}$.
\end{itemize}
\end{lemma}

\begin{proof}
Consider $(t_{k})_{k\in\mathbb{N}}\subset (m_{k})_{k\in\mathbb{N}}$ and $(h_{t_k})_{k\in\mathbb{N}}\subset V$ as in Lemma \ref{lemmamain}. Let $K$ be the basic constant of the basic sequence $(h_{t_k})_{k\in\mathbb{N}}$ and let $0<\epsilon<\frac{1}{2K}$. $($Recall that $K$ is always equal or bigger than 1, therefore $\epsilon<1$ $)$. Let $s_{1}=t_{1}$.  Suppose defined, by induction, $\{s_1,\ldots, s_n\}\subset \{t_1,t_2,\ldots\}$. Since $\displaystyle\lim_{j\rightarrow\infty} |h_{s_1}(t_j)|+\ldots+|h_{s_n}(t_j)|=0$, exist $s_{n+1}\in \{t_1,t_2,\ldots\}$, $s_{n+1}>s_n$, such that $$|h_{s_1}(s_{n+1})|+\ldots+|h_{s_n}(s_{n+1})|\leq \frac{\epsilon}{2^{n+1}\times 8}.$$ 
The induction to construct $(s_k)_{k\in\mathbb{N}}\subset (m_k)_{k\in\mathbb{N}}$ is completed.\\

Now define $l_{0,k}=h_{s_k}$. Notice that $l_{0,k}(s_k)=1$ and $l_{0,k}(s_j)=0$ for  $s_j\in\{s_1,...,s_k\}\setminus\{s_k\}$. Define $l_{1,k}=l_{0,k}-l_{0,k}(s_{k+1})h_{s_{k+1}}$.

Notice that
\begin{itemize}
\item $l_{1,k}(s_k)=1$ and $l_{1,k}(s_j)=0$ for  $s_j\in\{s_1,...,s_{k+1}\}\setminus\{s_k\}$,
\item since $|h_{s_{k}}(s_{k+1})|\leq \frac{\epsilon}{2^{k+1}\times 8}$ then 
$$|l_{1,k}(s_j)|\leq |l_{0,k}(s_j)|+|h_{s_{k+1}}(s_j)|=|h_{s_k}(s_j)|+|h_{s_{k+1}}(s_j)|$$ for every $j\in\mathbb{N}$.
\item $|l_{1,k}-l_{0,k}|_{\infty}=|l_{0,k}(s_{k+1})| |h_{s_{k+1}}|_{\infty}\leq\frac{\epsilon}{2^{k+1}\times 8}\times 8=\frac{\epsilon}{2^{k+1}}$\\
\end{itemize}

Suppose we have already defined, by induction, $l_{0,k},l_{1,k},\ldots,l_{t,k}\in V$ such that
\begin{itemize}
\item $l_{n,k}(s_k)=1$ for $0\leq n\leq t$,
\item $l_{n,k}(s_j)=0$ for $s_j\in\{s_1,...,s_{k+n}\}\setminus\{s_k\}$ and $0\leq n\leq t$,
\item $|l_{n,k}(s_j)|\leq |h_{s_k}(s_j)|+|h_{s_{k+1}}(s_j)|+\ldots+|h_{s_{k+n}}(s_j)|$, for every $j\in\mathbb{N}$ and $0\leq n\leq t$,
\item $|l_{n,k}-l_{n-1,k}|_{\infty}\leq\frac{\epsilon}{2^{k+n}}$ for $1\leq n\leq t$,\\
\end{itemize}
 
Next, define $l_{t+1,k}=l_{t,k}-l_{t,k}(s_{k+t+1}) h_{s_{k+t+1}}$. Notice that 
 
\begin{itemize}
 \item $l_{t+1,k}(s_k)=1$, 
 \item $l_{t+1,k}(s_j)=0$ for $s_j\in\{s_1,...,s_{k+t+1}\}\setminus\{s_k\}$,
 \item since $|l_{t,k}(s_{k+t+1})|\leq |h_{s_k}(s_{k+t+1})|+|h_{s_{k+1}}(s_{k+t+1})|+\ldots+|h_{s_{k+t}}(s_{k+t+1})|\leq$
 $$|h_{s_1}(s_{k+t+1})|+|h_{s_2}(s_{k+t+1})|+\ldots+|h_{s_{k+t}}(s_{k+t+1})|\leq \frac{\epsilon}{2^{k+t+1}\times 8}$$ then  $|l_{t+1,k}(s_j)|\leq |l_{t,k}(s_j)|+|h_{s_{k+t+1}}(s_j)|$ for every $j\in\mathbb{N}$ and by induction hypothesis
 $$|l_{t+1,k}(s_j)|\leq |h_{s_k}(s_j)|+|h_{s_{k+1}}(s_j)|+\ldots+|h_{s_{k+t+1}}(s_j)|,$$ for every $j\in\mathbb{N}$.
 \item $|l_{t+1,k}-l_{t,k}|_{\infty}=|l_{t,k}(s_{k+t+1})||h_{s_{k+t+1}}|_{\infty}\leq\frac{\epsilon}{2^{k+t+1}\times8}\times 8= \frac{\epsilon}{2^{k+t+1}}$
\end{itemize}
 
The induction to construct $(l_{t,k})_{t=0}^{\infty}\subset V$, for every $k\in\mathbb{N}$, is completed.\\

Now $|l_{0,k}|_{\infty}+|l_{1,k}-l_{0,k}|_{\infty}+|l_{2,k}-l_{1,k}|_{\infty}+\ldots\leq |l_{0,k}|_{\infty}+\frac{\epsilon}{2^{k+1}}+\frac{\epsilon}{2^{k+2}}+\ldots \leq |l_{0,k}|_{\infty}+\epsilon $. Thus, for each $k\in\mathbb{N}$, the series $\displaystyle\lim_{t\rightarrow\infty} l_{t,k}=l_{0,k}+(l_{1,k}-l_{0,k})+(l_{2,k}-l_{1,k})+\ldots$ is absolutely and coordinatewise convergent to some $l_k\in V$. Notice that $l_{t,k}(s_k)=1$ for every $t$ then $\displaystyle\lim_{t\rightarrow\infty} l_{t,k}(s_k)=l_k(s_k)=1$. Next $l_{t,k}(s_j)=0$ for $t>j$ and $j\neq k$ then $\displaystyle\lim_{t\rightarrow\infty} l_{t,k}(s_j)=l_k(s_j)=0$. Now, $l_{t,k}-l_{0,k}=(l_{t,k}-l_{t-1,k})+\ldots+(l_{1,k}-l_{0,k})$ then $$|l_{t,k}-l_{0,k}|_{\infty}\leq\frac{\epsilon}{2^{k+t}}+\frac{\epsilon}{2^{k+t-1}}+\ldots+\frac{\epsilon}{2^{k+1}}\leq \frac{\epsilon}{2^{k}},$$
then $\displaystyle\lim_{t\rightarrow\infty}|l_{t,k}-l_{0,k}|_{\infty}=|l_k-h_{s_k}|_{\infty}\leq \frac{\epsilon}{2^{k}}$,  for every $k\in\mathbb{N}$, so  $$|l_k|_{\infty}\leq |h_{s_k}|_{\infty}+\frac{\epsilon}{2^k} \leq 8+1=9.$$

Since $h_{s_k}(s_k)=1$ then $|h_{s_k}|_{\infty}\geq 1$ and we have $$\left|\frac{l_k}{|h_{s_k}|_{\infty}}-\frac{h_{s_k}}{|h_{s_k}|_{\infty}}\right|_{\infty}\leq \frac{\epsilon}{2^{k}}.$$ 
 
Then $\delta=\displaystyle\sum_{k=1}^{\infty} \left|\frac{l_k}{|h_{s_k}|_{\infty}}-\frac{h_{s_k}}{|h_{s_k}|_{\infty}}\right|_{\infty}\leq \sum_{k=1}^{\infty} \frac{\epsilon}{2^{k}}= \epsilon$.

Now the normalized sequence $\left(\frac{h_{s_k}}{|h_{s_k}|_{\infty}}\right)_{k\in\mathbb{N}}$ as a block basis of the basic sequence $(h_{t_k})_{k\in\mathbb{N}}$ is also a  basic sequence with basic constant $K'\leq K$.  Then $2K'\delta\leq 2K\delta\leq 2K\epsilon<1$.

By the {\em principle of small pertubation} \cite[Theorem 4.5]{carothers} the sequence $\left(\frac{l_k}{|h_{s_k}|_{\infty}}\right)_{k\in\mathbb{N}}$ is a basic sequence equivalent to the normalized basic sequence $\left(\frac{h_{s_k}}{|h_{s_k}|_{\infty}}\right)_{k\in\mathbb{N}}$.  Notice that $(l_k)_{k\in\mathbb{N}}$ is a  block basis of $\left(\frac{l_k}{|h_{s_k}|_{\infty}}\right)_{k\in\mathbb{N}}$, therefore is also a basic sequence. Finally define $l_{s_k}=l_k$.
\end{proof}

From the previous lemma, we can now infer the following.

\begin{proposition}\label{prop1}
Let $V$ be an infinite dimensional closed subspace of $\ell_{\infty}$. There exists $0 \neq h \in V \setminus Z(V)$.
\end{proposition}

\begin{proof}
Consider $l_{s_k}$ from Lemma \ref{lemmamain2}. We have that $l_{s_k}\in V\setminus Z(V)$.
\end{proof}

\begin{corollary}\label{spa}
$Z(\ell_\infty)$ is not spaceable in $\ell_{\infty}$.
\end{corollary}

As a consequence of Lemma \ref{lemmamain2} we also have the following result, whose proof is simple.

\begin{corollary} 
Let $V$ be an infinite dimensional closed subspace of $c_0$. Then $V \setminus Z(V)$ is dense in $V$.
\end{corollary}

\begin{proof}
Every $0\neq f\in V \subset c_0$, satisfies $\displaystyle\lim_{k\rightarrow\infty}f(n_k)=0$.
Let $\epsilon>0$. There exists $(m_k)_{k\in\mathbb{N}}\subset (n_k)_{k\in\mathbb{N}}$ such that $(f(m_k))_{k\in\mathbb{N}}\in l_1$ and $|(f(m_k))_{k\in\mathbb{N}}|_1\leq\dfrac{\epsilon}{9}$.

By Lemma \ref{lemmamain2}, exist $(s_k)_{k\in\mathbb{N}}\subset (m_k)_{k\in\mathbb{N}}$ and $l_{s_k}\in V$ such that 
\begin{itemize}
\item[$a)$] $l_{s_k}(s_k)=1$,
\item[$b)$] $l_{s_k}(s_j)=0$, for $j\neq k$.
\item[$c)$] $|l_{s_k}|_{\infty}\leq 9$ for every $k\in\mathbb{N}$.
\end{itemize}
Notice that $|f(s_1)l_{s_1}|_{\infty}+|f(s_2)l_{s_2}|_{\infty}+\ldots \leq (|f(s_1)|+|f(s_2)|+\ldots)\ 9\leq  \epsilon.$\\

Therefore $f-f(s_1)l_{s_1}-f(s_2)l_{s_2}-\ldots$ converge absolutely and coordinatewise to some $g\in V$. Notice that  for every $k\in\mathbb{N}$
$$g(s_k)=f(s_k)-f(s_1)l_{s_1}(s_k)-f(s_2)l_{s_2}(s_k)-\ldots=f(s_k)-f(s_k)l_{s_k}(s_k)=0$$
and $|g-f|_{\infty}\leq\epsilon$.
\end{proof}

\section{Algebrability and maximal spaceability of $\ell_p \setminus Z(\ell_p)$, for $p\in [1,\infty]$.}

In this section we prove that, although $Z(\ell_p)$ is not spaceable in $\ell_p$, for every $p\in [1,\infty]$, we have that $V\setminus Z(V)$ is maximal spaceable for every infinite dimensional closed subspace $V$ of $\ell_{p}$, for $p\in [1,\infty]$. The spaceability, for the case $p\in [1,\infty[$, shall be obtained in a strong sense, meaning that the Banach space constructed inside $V\setminus Z(V)$ shall be complemented in $\ell_p$. Also, $V\setminus Z(V)$ is maximal algebrable provided $V$ is any infinite dimensional closed subalgebra of $\ell_{p}$, $p\in [1, \infty]$.

\begin{theorem} \label{result}
Let $V$ be an infinite dimensional closed subspace of $\ell_p$, $p\in [1,\infty]$. Then $V\setminus Z(V)$ is maximal spaceable in $V$.
\end{theorem}

\begin{proof} 
By Lemmas \ref{lemmaB} and \ref{lemmamain2}, there is an increasing sequence of natural numbers $(s_{k})_{k\in\mathbb{N}}$ and a sequence $(l_{s_k})_{k\in\mathbb{N}}\subset V$ such that 
\begin{itemize}
\item[$a)$] $l_{s_k}(s_k)\neq 0$,
\item[$b)$] $l_{s_k}(s_j)=0$, for $j\neq k$.
\end{itemize}

Let $W=\langle l_{s_2},l_{s_4},l_{s_6},\ldots\rangle $ and notice that every $f\in W$ satisfies $f(s_{2k-1})=0$ for every $k\in\mathbb{N}$. Since convergence in norm implies coordinatewise convergence in $\ell_{p}$, $p\in [1,\infty] $ then for every $f\in\overline{W}$, we obtain $f(s_{2k-1})=0$ for every $k\in\mathbb{N}$.

Notice that $\{l_{2k}\in\overline{W},\ k\in\mathbb{N}\}$ is a linear independent set then $\overline{W}$ is a infinite dimensional closed subspace of $V$ with $\overline{W}\subset V\setminus Z(V)\cup\{0\}$.
\end{proof}

\begin{corollary} 
Let $V$ be an infinite dimensional closed subspace of $\ell_p$, $p\in [1,\infty[$. Then the infinite dimensional closed subspace $\overline{W}\subset V\setminus Z(V)\cup\{0\}$, obtained in Theorem \ref{result}, is complemented in $\ell_p$.
\end{corollary}

\begin{proof}
Notice that the sequence $(l_{s_k})_{k\in\mathbb{N}}\subset V$  used in the proof of theorem \ref{result} is a  basic sequence such that $[l_{s_1},l_{s_2},\ldots]$ is complemented in $\ell_p$.  Since $\overline{W}=[l_{s_2},l_{s_4},l_{s_6},\ldots]$ is complemented in $[l_{s_1},l_{s_2},\ldots]$ by $[l_{s_1},l_{s_3},l_{s_5},\ldots]$. We got the result.
\end{proof}

\begin{theorem}\label{result2}
Let $V$ be an infinite dimensional closed subalgebra of $\ell_{p}$, $p\in [1, \infty]$, with the coordinatewise product. Then $V\setminus Z(V)$ is algebrable in $V$.
\end{theorem}

\begin{proof}
By Lemmas \ref{lemmaB} and \ref{lemmamain2}, there is an increasing sequence of natural numbers $(s_k)_{k\in\mathbb{N}}$ and $l_{s_k}\in V$ such that 
\begin{itemize}
\item[$a)$] $l_{s_k}(s_k)\neq 0$,
\item[$b)$] $l_{s_k}(s_j)=0$, for $j\neq k$.
\end{itemize}

Consider the following closed subalgebra of $V$ :  $$V(0,(s_{2k-1})_{k\in\mathbb{N}})=\{f\in V,\text{ such that } f(s_{2k-1})=0, k\in\mathbb{N}\}\subset V\setminus Z(V)\cup\{0\}.$$

Now $V(0,(s_{2k-1})_{k\in\mathbb{N}})$ is an infinite dimensional subspace of $V$  since $\{l_{s_{2k}},\ k\in\mathbb{N}\}$ is a linear independent subset. Since any infinite dimensional closed algebra has at least $\mathfrak{c}$ generators then $V\setminus Z(V)$ is algebrable.
\end{proof}

\subsection*{Acknowledgements}
\noindent  D. Cariello was supported by CNPq-Brazil Grant 245277/2012-9. 


\end{document}